  \documentclass{amsart}
\usepackage{amsmath, amssymb,epic,graphicx,mathrsfs,enumerate}

\usepackage[all]{xy}

\usepackage[dvipsnames]{xcolor}
\usepackage{amsthm}
\usepackage{amssymb}
\usepackage{latexsym}
\usepackage{longtable}
\usepackage{epsfig}
\usepackage{amsmath}
\usepackage{hhline}

\newtheorem{thm}{Theorem}
\newtheorem{cor}[thm]{Corollary}

 \newtheorem{lemma}[thm]{Lemma}

\DeclareMathOperator{\aut}{Aut}

\DeclareMathOperator{\diam}{diam}

\newcommand{\w}{\widetilde}

\renewcommand{\emptyset}{\varnothing}

\newcommand{\ff}{\mathfrak{F}}

\newcommand{\G}{\Gamma}

\newcommand{\gf}{\G_\ff}
\newcommand{\iso}{\mathcal{I}_\ff}

\newcommand{\iss}{\mathcal{I}(G)}
\newcommand{\isss}{\mathcal{I}}

\begin{document}
	\bibliographystyle{amsplain}
	\title[The non-two-primes graph of a finite group]{The non-two-primes graph of a finite group}
	\author{Karmele Garatea-Zaballa }
	\address{Dipartimento di Matematica \lq\lq Tullio Levi Civita\rq\rq,  Universit\`a 
	di Padova, Via Trieste 63, 35121 Padova, Italy}\email{karmele.garateazaballa@studenti.unipd.it}
	\author{Andrea Lucchini}
	\address{Dipartimento di Matematica \lq\lq Tullio Levi Civita\rq\rq,  Universit\`a 
		di Padova, Via Trieste 63, 35121 Padova, Italy}\email{lucchini@math.unipd.it}

	\begin{abstract}To any finite group $G$, we may associate a graph whose vertices are the elements of $G$ and where two distinct vertices $x$ and $y$ are adjacent if and only if the order of the subgroup $\langle x, y\rangle$ is divisible by at least 3 distinct primes. We prove that the subgraph of this graph induced by the non-isolated vertices is connected and has diameter at most 5.
		\end{abstract}
	\maketitle
\section{Introduction}
Let $\mathfrak F$ be a class of finite groups and $G$ a finite group. We may consider a graph $\w{\G}_\ff(G)$  whose vertices are the elements of $G$ and where two  vertices $g,h\in G$  are connected if and only if $\langle g,h\rangle\notin\ff$. We denote by $\iso(G)$ the set of isolated vertices of $\w{\G}_\ff(G)$. We define the non-$\ff$-graph $\gf(G)$ of $G$  as the subgraph of $\w{\G}_\ff(G)$ obtained by deleting the isolated vertices. 
In the particular case when $\mathfrak F$ is the class $\mathfrak A$ of the abelian groups, the graph $\Gamma_{\mathfrak A}(G)$ has been introduced by Erd\"{o}s and it is known with the name of non-commuting graph (see for example\cite{ncg}, \cite{neu}). If $\mathfrak F$ is the class $\mathfrak N$ of the finite nilpotent groups, then $\Gamma_{\mathfrak N}(G)$ is the non-nilpotent graph, studied for example in \cite{az}. When $\mathfrak F$ is the class $\mathfrak S$ of the finite soluble groups, we obtain the non-soluble graph (see \cite{ns}). The subset $\iso(G)$ is not in general a subgroup of $G$, however this occurs for several saturated formations (see for example \cite[Theorems 1.1 and 1.3]{nof}) and, more in general, we call semiregular a class $\ff$ with the property
that $\iso(G)$ is a subgroup of $G$ for every finite group $G$.
Moreover we call connected a class $\ff$ with the property that the graph $\gf(G)$ is connected for any finite group $G$. The results obtained in \cite{nof} indicate that often a semiregular formation is connected. This occurs for example for the formations of abelian groups, nilpotent groups, soluble groups, supersoluble groups, groups with nilpotent derived subgroup, groups with Fitting length less or equal than $t$ for any $t\in \mathbb N.$ The question whether a semiregular class is necessarily connected has been investigated in \cite{ln}. The answer is negative, however the results presented in \cite{ln} indicate that in many significant cases the connectivity of a class can be proved using its semiregularity property. 

In this paper we investigate the connectivity properties of the non-$\ff$-graph when $\ff$ is the class of the finite groups whose order is divisible by at most two different primes. Notice that by the Burnside’s $p^aq^b$-theorem, this class is contained in the class of finite solvable group. It is interesting to notice that the class $\ff$ is not semiregular. For example, if $G=\langle a, b \mid a^{15}=1, b^2=1, (ab)^2=1\rangle$ is the dihedral group of order 30, then 
 $\iso(G)=\{1,a^3,a^5,a^6,a^9,a^{10},a^{12}\}$ is not a subgroup of $G.$ Thus the methods developed in \cite{nof} and \cite{ns} cannot help to investigate whether the non-$\ff$-graph is connected and ad-hoc arguments are needed. Our main result is the following:
\begin{thm}\label{main}
Let $\ff$ be the class of the finite groups whose order is divisible by at most two different primes. For every finite group $G$, the non-$\ff$-graph $\gf(G)$ is connected and its diameter is at most 5.
\end{thm}

We don't know whether the bound on the diameter is sharp but we may exhibit an example of a group $G$ such that $\gf(G)$ has diameter at least 3. Indeed let
$H={\rm{SL}}(2,3).$ Then $H$ has a faithful irreducible action on $A=C_3\times C_3$ and a non-trivial action on $B=C_7.$ Let $G=(A\times B)\rtimes H$ be the semidirect product with respect of these actions and let $x$ and $y$ be two elements of $G$ of order, respectively, 7 and 2. If an element $g$ of $G$ is adjacent to $x$ in $\gf(G)$ then $|g|=6,$ while if $g$ is adjacent to $y$ then $|g|\in \{14,21,28\},$ so the distance in $\gf(G)$ between $x$ and $y$ is at least 3.

\section{Notations and preliminary results}

Given a finite group $G$, we will denote by $\Gamma(G)$ the graph $\gf(G),$ being $\ff$ the class of the finite groups whose order is divisible by at most two different primes. We will denote by $\iss$
the set of the isolated vertices of $\Gamma(G)$, by $\diam(\Gamma(G))$ the diameter of $\Gamma(G)$ and by $d(x,y)$ the distance between two vertices $x, y$ of $\Gamma(G).$ Moreover we will denote by $\pi(G)$ the set of the prime divisors of the order of $G$ and by $\tilde\pi(G)$ the cardinality of $\pi(G).$ In a similar way, if $g \in G$, then $\pi(g)$ will denote the set of the prime divisors of the order of $g$ and $\tilde\pi(g)$ the cardinality of this set. With $\Sigma(G)$ we will denote the set of the elements $g$ of $G$ with $\tilde\pi(g)\geq 2.$
Finally, for every product $n=p_1\cdots p_t$ of distinct primes, we will denote by $\Omega_n(G)$ the set of elements $g\in G \setminus \iss$ whose order is divisible by $n$, but not by any prime $q\notin \{p_1,\dots,p_t\}.$

\

The following result by Higman will play a crucial role in our proofs.

\begin{thm}\cite[Theorem 1]{hig}
	Let $G$ be a finite solvable group all of whose elements have prime power order. Then $G$ has order divisible by at most two primes.
\end{thm}

\begin{cor}\label{higman}
If $G$ is a finite solvable group and $\iss\neq G,$ then $\Sigma(G)\neq \emptyset.$
\end{cor}

\begin{lemma}
	\label{rdivides}
	Let $G$ be a finite group and let $N$ be a non-trivial normal subgroup of $G$
	such that $|N|=p^a$ for some prime $p.$ For every pair $x_1, x_2$ of elements of $G,$ there exist $n_1,n_2\in N$ such that $p$ divides the order of $\langle x_1n_1, x_2n_2\rangle.$
\end{lemma}

\begin{proof}
It is not restrictive to assume that $G=\langle x_1, x_2\rangle N$ and that $N$ is a minimal normal subgroup of $G$. Let $H=\langle x_1,x_2\rangle.$ We may assume that $p$ does not divide $|H|$, which implies that $H$ is complement of $N$ in $G.$ Assume that our statement is false. Then $\langle x_1n_1, x_2n_2\rangle$ is a complement of $N$ in $G$ for every $n_1, n_2\in N.$ Moreover if $\langle x_1n_1, x_2n_2\rangle=\langle x_1m_1, x_2m_2\rangle$ with $n_1, n_2, m_1, m_2\in N,$ then $n_1=m_1$ and $n_2=m_2.$
Consequently, there are $\vert N \vert^2$ complements for $N$ in $G$. However, by Schur-Zassenhaus theorem, all these complements are conjugate to $H$, so $|N|^2\leq |G:N_G(H)|\leq |G:H|=|N|,$ a contradiction. 
\end{proof}

Let $A$ and $N$ be finite groups, and suppose that $A$ acts on $N$ via automorphism. The action of $A$ on $N$ is said to be Frobenius if $n^a\neq n$ whenever $n\in N$ and $a\in A$ are nonindentity elements. 
Equivalently, the action of $A$ on $N$ is Frobenius if and all if $C_N(a)=1$ for all nonindentity elements $a\in A.$


\begin{lemma}\label{fpf}Let $N$ be a non-trivial normal subgroup of a finite group $G$ and $x,y \in G.$ If $N$ is a $p$-group and $C_N(x)=1,$ then there exists $n\in N$ such that $p$ divides the order of $\langle x, yn\rangle.$
\end{lemma}
\begin{proof}
	By Lemma \ref{rdivides}, there exists $n_1,n_2\in N$ such that $p$ divides  the order of $\langle xn_1, yn_2\rangle.$ Since $C_N(x)=1,$ there exists $n\in N$ such that $x=(xn_1)^m$. But then $p$ divides the order of $\langle (xn_1)^m, (yn_2)^m\rangle=
	\langle x, yn\rangle,$ with $n=[y,m]n_2^m.$
\end{proof}

\begin{lemma}\label{caso3}
	If there exists $g\in G$ with $\tilde \pi(g)\geq 3,$ then $\iss=\emptyset$ and ${{\rm{diam}}(\Gamma(G)})\leq 2.$
\end{lemma}	

\begin{proof}
	If $\tilde \pi(g)\geq 3,$ then every other element of $G$ is adjacent to $g$ in $\Gamma(G).$ This implies ${{\rm{diam}}(\Gamma(G)})\leq 2.$
\end{proof}

\begin{lemma}\label{twodiv} If $x, y \in \Sigma(G)$ and $\tilde \pi(G)\geq 3,$ then $d(x,y)\leq 2.$
	\end{lemma}
	\begin{proof}By the previous lemma we may assume $\pi^*(g)=2$ for every $g\in \Sigma(G).$
	If $\pi(x)\neq \pi(y),$ then $\pi(x)\cup \pi(y)$ contains at least 3 different primes and consequently $\tilde \pi(\langle x, y\rangle)\geq 3,$ and $x$ and $y$ are adjacent vertices of $\Gamma(G).$ If $\pi(x)=\pi(y),$ then there exists a prime $p$ in $\pi(G)\setminus \pi(x).$ If $z$ is an element of $G$ of order $p,$ then
	$x-z-y$ is a path in $\Gamma(G).$
	\end{proof}
	
\begin{lemma}
	\label{indbound}
	Let $G$ be a finite solvable group with $\iss\neq G$. Then for every $x\notin \iss$ there exists $y\in \Sigma(G)$ such that  $d(x,y)\leq 2$.
\end{lemma}

\begin{proof}
	We  prove the statement by induction on $\vert G \vert$. Let $x\in G\setminus \iss$. If $\tilde\pi(x)\geq 2$ the statement is clear, so we may assume that $\tilde\pi(x)=1$. As $x \notin \iss$, there exists $y\in G$ such that $x$ and $y$ are adjacent in $\Gamma(G).$ If $\langle x,y\rangle \neq G$, then by induction there exists $z\in \langle x,y \rangle$, with $\tilde \pi(z)\geq 2$, such that $d(x,z)\leq 2$. So we may assume that $G=\langle x,y \rangle$.

	Let $N$ be a minimal normal subgroup of $G$. There exists a prime $r$ such that $N$ is an elementary abelian $r$-group. 
	Notice that for every $1\neq n\in N$ and $g\in G,$ $\pi(\langle n,g\rangle)=\{r\}\cup \pi(g)$. Thus either $N\subseteq \iss$ or $G$ contains an element $\tilde g$ whose order is divisible by two primes in $\pi(G)\setminus \{r\}$ and all the non-trivial elements of $N$ are adjacent to $\tilde g\in \Gamma(G).$
	In particular we may assume $x\notin N.$

	 If $\tilde \pi(G/N) > 2$, 
	 then, since $G/N=\langle xN, yN\rangle,$ it follows that $xN \notin \isss(G/N)$, so by induction $d(xN,zN)\leq 2$, and consequently $d(x,z)\leq 2,$ for some element $z$ with $2\leq \tilde\pi(zN)\leq \tilde\pi(z).$ So we may assume $\pi(G/N)=\{p,q\}$ with $r\notin \{p,q\}.$
	Moreover, without loss of generality, we may assume that $x$ is a $p$-element. If there exists $1\neq n\in N$ such that $[x,n]=1$, then  $\pi(xn)=\{p,r\}$ and $xn$ is adjacent in $\Gamma(G)$ to $y$, implying $d(x,xn)=2$. So we may assume that $C_N(x)=1.$ By Corollary \ref{higman}, 
there exists an element $z\in G$ with $\tilde\pi(z)\geq 2$. If $pq$ divides $|z|,$ then by Lemma \ref{fpf} there exist $n\in N$ such that $r$ divides the order of $\langle x, zn\rangle.$ Thus $x$ is adjacent in $\Gamma(G)$ to the element $zn,$ whose order is divisible by $pq.$ If $qr$ divides $|z|$ then $x$ and $z$ are adjacent in $\Gamma(G).$ Finally, assume that $pr$ divides $|z|.$ In this case let $v$ be an element of $G$ of order $q.$
By Lemma \ref{fpf} there exists $n\in N$ such that $r$ divides the order of $\langle x,vn \rangle,$ and $x-vn-z$ is a path in $\Gamma(G).$
\end{proof}

Recall that the prime graph $\Pi(G)$ of a finite group is the graph whose vertices are the primes dividing the order of $G$ and where two vertices $p$ and $q$ are joined by an edge if there is an element in $G$ of order $pq.$

\begin{lemma}\label{long} Let G be a finite solvable group with $G\neq \iss$ and $\tilde\pi(G) = 3.$  If ${{\rm{diam}}(\Gamma(G)})>4,$ then
$\Pi(G)$ is a path graph of length three.
\end{lemma}

\begin{proof} Let $\pi(G)=\{p,q,r\}$ and assume $\diam(\Gamma(G))>4.$ Our aim is to prove that, for a suitable labeling of the three vertices, the only edges of the prime graph $\Pi(G)$ of $G$ are $(p,r)$ and $(r,q).$
To that end we will see that if $\Pi(G)$ is the complete graph $K_3$ or if it contains an isolated vertex, then ${{\rm{diam}}(\Gamma(G)})\leq 4.$

It can be easily seen that if $\Pi(G)$ is complete, then $\diam(\Gamma(G))\leq 3$.
So we may assume that one of the vertices of $\Pi(G),$ for example $r, $ is isolated. In this case $\Pi(G)$ has two components: $\{p,q\}$ and $\{r\}.$ It follows from a result of Gruenberg and Kegel (see \cite[Corollary]{pg}) that
$G$ is Frobenius or 2-Frobenius and  one of the
two components consists of the primes dividing the lower Frobenius complement.
We distinguish two cases:
\begin{enumerate}
\item\label{primo} $G$ is a Frobenius group. In this case $G=N\rtimes H$ where $N$ is nilpotent and the action of $H$ on $N$ is Frobenius.

\item $G$ has normal subgroups $N$ and $K$ such that $K$ is a Frobenius group with Frobenius kernel $N$, and $G/N$ is a Frobenius group with Frobenius kernel $K/N$.
\end{enumerate}

In case (1), we have two possibilities:

a) $N$ is a $\{p,q\}$-group and $H$ is an $r$-group.

b) $N$ is an $r$-group and $H$ is a $\{p,q\}$-group.

In case (a) all the elements of $G\setminus N$ are $r$-elements and the $p$-elements and $q$-elements of $G$ belong to $N$ and are isolated in $\Gamma(G).$ So if $g\in G\setminus \iss,$ either $g$ is an $r$-element or $pq$ divides $|g|.$ This implies that  $\diam(\Gamma(G))\leq 2$.

 In case (b) every element of $\Omega_{pq}(G)$ is adjacent to every element of $\Omega_r(G)$. We need to investigate the behavior of the elements in $\Omega:=\Omega_p(G)\cup \Omega_q(G).$ Let $x\in \Omega.$ By Corollary \ref{higman}, there exists $z\in \Omega_{pq}(G).$
Since $N$ is the Frobenius kernel of $G$ and $x\notin N,$ we have $C_N(x)=1$ so, by Lemma \ref{fpf}, there exists $n$ in $N$
 such that $r$ divides the order of $\langle x, zn\rangle.$ Thus $x$ is adjacent to $zn \in \Omega_{pq}(G).$ We have so proved that every element in $\Omega$ is adjacent in $\Gamma(G)$ to some element in $\Omega_{pq}(G).$ By Lemma \ref{twodiv}, we can conclude that $\diam(\Gamma(G))\leq 4$.

Let us move to case (2). Again we have two possibilities:

a) $K/N$ is a nilpotent $\{p,q\}$-group and $G/K$ and $N$ are $r$-groups.

b) $K/N$ is an $r$-group and $G/K$ and $N$ are $\{p,q\}$-groups.

In case (a) we may repeat the previous argument: all the elements of $\Omega:=\Omega_p(G)\cup \Omega_q(G)$ belong to $K$ and are adjacent to an element of $\Omega_{pq}(G).$ 

In case (b) we start by noticing that
$\Omega\cap N=\emptyset.$ First assume that $G$ contains an element $z$ with $|zK|=pq.$ Let $x\in \Omega.$ It must be $C_{K/N}(xK)=1,$ otherwise $G$ would contain an element of order $rt$, being $t$ the prime in $\{p,q\}$ dividing the order of $x$. So by Lemma \ref{fpf}, there exists $k\in K$ such that $r$ divides the order
of $\langle x,zk\rangle N/N$. In particular $x$ and $zk$ are adjacent vertices in $\Gamma(G).$ We have so proved that every element of $\Omega$ is adjacent to a vertex in $\Omega_{pq}(G)$ and this suffices to conclude that $\diam(\Gamma(G))\leq 4$.

 To conclude our discussion of case (b), we remain with the case when $G/K$ does not contain elements of order $pq.$
 
  We claim that in this case $\tilde\pi(G/K)=1.$ To prove this, recall that every Sylow subgroup of a Frobenius complement is cyclic or generalized quaternion (see for example \cite[Corollary 6.17]{Isaacs}).
This implies in particular that the Frobenius complement $K/N$ is either a cyclic $r$-group or a generalized quaternion group. If $K/N$ is cyclic, then  the Frobenius complement $G/K$,
being isomorphic to a subgroup of $\aut(K/N)$, is an abelian group; its Sylow subgroups must be cyclic and therefore $G/K$ is a cyclic $\{p,q\}$-group. However we are assuming that $G/K$ does not contain elements of order $pq$ so $G/K$ is a $p$-group or a $q$-group.
If $K/N$ is generalized quaternion, then $G/K$ is a Frobenius complement of odd order. Hence, by \cite[Lemma 2.4]{parker},  any two elements in $G/K$  of coprime order commute, so if $p$ and $q$ both divide the order of $G/K$, then there must be an element in $G/K$ whose order is divisible by $pq$, against our assumption.

So our claim is proved and we may assume that $G/K$ is a $p$-group. In this case all the $q$-elements of $G$ belong to $N$ and are isolated in $\Gamma(G).$ Thus the vertex set of $\Gamma(G)$ coincides
with $\Omega_p(G) \cup \Omega_r(G) \cup \Omega_{pq}(G).$ 

Now we claim that every element $g\in \Omega_p(G)$ is adjacent to an element in $\Omega_r(G) \cup \Omega_{pq}(G).$ This would allow to conclude $\diam(\Gamma(G))\leq 4$.

In order to prove the claim, let $y$ be an element of $G$ order $r$. If $g$ is adjacent to $y$ we are done, so assume that $\langle g, y\rangle$ is a $\{p,r\}$-group.   
Now let $Q$ be the Sylow $q$-subgroup of $N$, and let $Z=Z(Q)$ be the center of $Q$. If $[gy,z]=1$ for some $1\neq z \in Z$, then $gyz \in \Omega_{pq}$ and the order of $\langle g,gyz \rangle =\langle g,yz \rangle$ is divisible by $pqr$. This implies that $g$ and $gyz$ are adjacent vertices of $\Gamma(G).$ So we may assume that $C_Z(gy)=1$. Since the action of $G/K$ over $K/N$ is Frobenius, $gy=g^xn$ for some $x\in G$ and $n \in N$ and, for any $1\neq z\in Z$, $[g,z]=([g^x,z^x])^{x^{-1}}=([gyn^{-1},z^x])^{x^{-1}}=[gy,z^x]^{x^{-1}}\neq 1$, hence $C_N(g)=1.$
By Lemma \ref{fpf} there exist $\tilde z \in Z$ such that $\langle g,y\tilde z \rangle$ has order divible by $q$. Thus implies that
$g$ is adjacent to the $r$-element $y\tilde z.$
\end{proof}

\section{Proof of the main result}

\begin{lemma}\label{soluno}If $G$ is soluble and $\tilde\pi(G)\geq 4,$ then  ${{\rm{diam}}(\Gamma(G)})\leq 3.$
\end{lemma}

\begin{proof}
By Lemma \ref{caso3}, we may assume $\tilde \pi(x)\leq 2$ for every $x\in G.$  If $p, q, r$ are three distinct primes in $\pi(G),$
then $G$ contains an element of order the product of two of these primes (see \cite[Proposition 1]{lucido}). In particular assume $\tilde\pi(G)\geq 4$ and let $x$ and $y$ be two distinct nonidentity elements of $G.$ Let $\pi=\pi(x) \cup \pi(y).$ If $|\pi|>2,$ then $x$ and $y$ are adjacent vertices of $\Gamma(G).$ So we may assume $|\pi|\leq 2.$ If there exists $p\in \pi(x)\cap \pi(y),$ then take
three different primes $q_1, q_2, q_3 \in \pi(G)\setminus \{p\}.$
Then $G$ contains an element $z$ whose order is the product of two of these primes and $x-z-y$ is a path is $\Gamma(G).$ So we may assume
$\pi(x)=\{p_1\}$ and $\pi(y)=\{p_2\},$ with $p_1\neq p_2.$ Let $r_1, r_2$ be two different primes in $\pi(G)\setminus \{p_1,p_2\}$. There exists in $G$ an element $z_1$ whose order is the product of two primes in $\{r_1,r_2,p_2\}.$ If
$|z_1|=r_1r_2,$ then $x-z-y$ is a path is $\Gamma(G).$ If $p_2$ divides $|z_1|$, then we may consider an element $z_2$ whose order is the product of two primes in $\{r_1,r_2,p_1\}.$ If $z_2=|r_1r_2|,$ then
$x-z-y$ is a path is $\Gamma(G).$ If $p_1$ divides $|z_2|,$ then $x-z_1-z_2-y$ is a path in $\Gamma(G).$
\end{proof}

\begin{lemma}\label{soldue}If $G$ is soluble and $\tilde\pi(G)=3,$ then 
${{\rm{diam}}(\Gamma(G)})\leq 5.$
\end{lemma}
\begin{proof}
 By Lemma \ref{long}, we may assume that the prime graph of $G$ is of the form $p-r-q$.

  Notice that if $x\notin \Omega_r(G),$ then $d(x,z)=1$ for some $z\in \Sigma(G)$. Indeed, if $p$ divides $\vert x\vert$ then $x$ is adjacent to an element in $\Omega_{rq}(G),$ while if $q$ divides  $|x|$ then $x$ is adjacent to an element in $\Omega_{rp}(G)$. If $x\in \Omega_r(G),$ then, by Lemma \ref{indbound}, $d(x,z)\leq 2$ for some $z\in \Sigma(G).$ Hence, using Lemma \ref{twodiv}, we conclude that if $x$ and $y$ are non-isolated vertices of $\Gamma(G)$ and $x$ is not an $r$-element, then $d(x,y)\leq 5$.

Assume now that $g\in \Omega_r(G)$ and that $g$ is adjacent to $x\notin \Omega_r(G).$ Let $y \notin \iss.$ By Lemma \ref{twodiv}, there exists $z\in \Sigma(G)$ with $d(z,y)\leq 2.$ If $x$ and $z$ are adjacent in $\Gamma(G),$ then $d(g,y)\leq d(g,x)+d(x,z)+d(z,y)\leq 4.$ If not, then $\pi(x)\subseteq \pi(z)$ and there exists $v\in \Sigma(G)$ with $\pi(v)\cup \pi(x)=\{p,q,r\}.$ But then $d(g,y)\leq d(g,x)+d(x,v)+d(v,z)+d(z,y)\leq 5.$

We remain with the case when $g$ belongs to the set $\Lambda$ of the elements of $\Omega_r(G)$ whose neighbourhood is contained in $\Omega_r(G).$ We are going to solve this case proving the following claim: there exist $a_1,a_2\in G\setminus \iss$ with the following properties:
\begin{enumerate}
	\item $d(g,a_1)=d(g,a_2)=2;$
	\item $\pi(a_1)\cup \pi(a_2)=\{p,q,r\}.$
\end{enumerate}
In order to prove the claim, let $g\in \Lambda$ and let $z\in \Omega_r(G)$ be adjacent to $g \in \Gamma(G).$ Let $H=\langle g,z\rangle$ and $N=O_r(H).$
Notice that neither $g$ nor $z$ belongs to $N$ since $\tilde\pi(\langle n,h\rangle)\leq 2$ for all $n\in N$ and $h\in H$.
Let $A/N$ be a minimal normal subgroup of $H/N$. Then $A/N$ is either a $p$-group or a $q$-group. Assume for example that it is a $p$-group
(a symmetric argument works if $A/N$ is a $q$-group). It must be $C_{A/N}(zN)=N/N.$ Indeed, if $[z,a]\in N$ for some $a\in A\setminus N$, then $rp$ divides $|za|$  and consequently $pqr$ divides $|\langle g,za\rangle|$ and $g$ is adjacent to $za$, in contradiction with the definition of $\Lambda.$
Notice also that it follows from Lemma \ref{fpf} that if $h$ is an $r$-element of $H$ and $C_{A/N}(hN)=N/N$, then $h$ is adjacent to an element of $H$ whose order is divisible by $q.$ In particular, since by definition all the elements of $G$ that are adjacent to $g$ are $r$-elements, $C_{A/N}(gN)\neq N/N$, so $[g,a]\in N$ for some $a\in A\setminus N$. 
But then $pr$ divides $|ga|$ and $\Gamma(G)$ contains the path $g-z-ga.$
Moreover from what we said above, since $C_{A/N}(zN)=N/N,$  the graph $\Gamma(G)$ contains a path $g-z-u$ with $q$ dividing the order of $u$. So we have proved our the claim.

We can now prove that if $y_1,y_2\in \Lambda$, then $d(y_1,y_2)\leq 5$. Choose $z_1,z_2\in \Sigma(G)$ such that $d(y_1,z_1)=d(y_2,z_2)=2$. If $\pi(z_1)\neq \pi(z_2),$ then $d(z_1,z_2)=1$ and therefore $d(y_1,y_2)=5$. If $\pi(z_1)=\pi(z_2),$ then, by our last claim, there exists $u$ such that $d(y_1,u)\leq 2$ and $\pi(u)\cup \pi(z_1)=\{p,q,r\}.$ But then $d(y_1,y_2)\leq d(y_1,u)+d(u,z_2)+d(z_2,y_2)\leq 5.$
\end{proof}

\begin{proof}[Proof of Theorem \ref{main}] Let $R=R(G)$ be the soluble radical of $G$ and consider the following subsets of $G:$
$$A=\isss(R)\setminus \iss,\quad B=R \setminus \isss(R),\quad C=G\setminus R.$$
Let us summarize a list of information concerning $\Gamma(G)$ that will be useful to prove our bound on the diameter of $\Gamma(G).$
\begin{enumerate}
\item[(i)]  $d(c_1,c_2)\leq 2$ for every $c_1,c_2\in C$. This follows from
\cite[Theorem 6.4]{gu}, stating that if $c_1,c_2 \notin R,$ then there exists $z\in G$ such that $\langle c_1,z\rangle$ and $\langle c_2,z\rangle$ are not solvable.
\item [(ii)] For every $a\in A,$ there exists $c\in C$ such that $d(a,c)=1$. This follows immediately from the definitions of $A$ and $C.$
\item [(iii)] If $B\neq\emptyset,$ then for every $b\in B,$ there exists $s\in \Sigma(G)\cap R$ with $d(b,s)\leq 2.$ This follows immediately from
 Lemma \ref{indbound}.

\item [iv)] $d(c,s)\leq 3$ for every $c\in C$ and $s\in \Sigma(G).$ If
$\tilde \pi(s)\geq 3,$ then $c$ and $s$ are adjacent in $\Gamma(G).$ So we may assume $\tilde\pi(s)=2.$ Since $G/R$ is not solvable, $\tilde \pi(G/R)\geq 3,$ so there exists $p\in \tilde \pi(G/R) \setminus \tilde \pi(s).$ Let $z$ be a $p$-element in $C.$ Then $d(c,z)\leq 2$ by $(i),$ so $d(c,s)\leq d(c,z)+d(z,s)\leq 3.$
\end{enumerate}
By (i) $\isss(G) \subseteq \isss(R),$ so $G\setminus \isss(G)$ is the disjoint union of the three subsets $A, B,C$. 
We can conclude our proof by analyzing the different possibilities:

\noindent If $a_1,a_2\in A$, then by (ii) there exist $c_1,c_2\in C$ with $d(a_1,c_1)=d(a_2,c_2)=1$. Hence, by (i), $d(a_1,a_2)\leq d(a_1,c_1)+d(c_1,c_2)+d(c_2,a_2)\leq 1+2+1=4$.
 
\noindent  If $a\in A$ and $b\in B$, then, by (ii), there exist $c\in C$ which is adjacent to $a$ in $\Gamma(G).$ Let $H=R\langle c\rangle.$
Notice that $H$ is solvable and $a$ and $b$ are non-isolated vertices of $\Gamma(H),$ so, by Lemmas \ref{soluno} and \ref{soldue}, $d(a,b)\leq 5.$


\noindent If $a\in A$ and $c\in C$, then, by (ii), there exists $c^*\in C$ such that $d(a,c^*)=1$, and, by (i), $d(c,c')\leq 2$, thus $d(a,c)\leq d(a,c^*)+d(c^*,c)\leq 3$.

\noindent If $b_1,b_2\in B$, then $b_1, b_2$ are non-isolated vertices of $\Gamma(R),$ so, by Lemmas \ref{soluno} and \ref{soldue}, $d(b_1,b_2)\leq 5.$

\noindent If $b\in B$ and $c\in C$, then, by (iii), there exists $s\in \Sigma(G)$ such that  $d(b,s)\leq 2$ and, by (iv), $d(b,c)\leq d(b,s)+d(s,c)\leq 2+3=5.$
\end{proof}

\end{document}